\documentclass[a4paper,10pt]{amsart}
\usepackage[english]{babel}

\usepackage[letterpaper,top=2cm,bottom=2cm,left=3cm,right=3cm,marginparwidth=1.75cm]{geometry}

\usepackage{amsmath}
\usepackage{graphicx}
\usepackage[colorlinks=true, allcolors=blue]{hyperref}
\usepackage{comment}
\usepackage{amsmath,amsthm,amssymb,amsfonts,amscd, amsbsy,mathtools}
\usepackage{float}

\usepackage{mathrsfs} %cursive letters

\newtheorem{remark}{Remark}[section]
\newtheorem{lemma}{Lemma}[section]
\newtheorem{theorem}{Theorem}[section]

\newtheorem{example*}{Example}
\newtheorem{remark*}{Remark}
\newtheorem{lemma*}{Lemma}
\newtheorem{theorem*}{Theorem}
\newtheorem{proposition*}{Proposition}
\newtheorem{corollary*}{Corollary}
\newtheorem{definition*}{Definition}

\title[Minimal surfaces and solitons to the MCF in $\mathbb{H}^3$ as translation surfaces]{Classification of minimal surfaces and solitons to the mean curvature flow in $\mathbb{H}^3$ as translation surfaces}

\author{Ferreira, T.A. and dos Santos, J.P.}

\address{Jo\~ao Paulo dos Santos - Departamento de Matem\'atica, 
Universidade de Bras\'ilia, 70910-900, Bras\'ilia-DF, Brazil}
\email{joaopsantos@unb.br}

\address{Tarcios Andrey Ferreira - Departamento de Matemática, Universidade de Brasília, 70910-900, Brasília-DF, Brazil }
\email{T.A.Ferreira@mat.unb.br}

\date{\today}

\begin{document}



\begin{abstract}
    We consider the hyperbolic three-space in the half-space model endowed with a metric Lie group structure. In this setting, translation surfaces are defined as products of two curves $\alpha$ and $\beta$ with respect to the Lie group operation. We investigate minimal surfaces and solitons to the mean curvature flow arising from specific types of products of these curves. In particular, we provide classification results for minimal surfaces, hyperbolic translators, and conformal solitons to the mean curvature flow.
\end{abstract}

\maketitle

\section{Introduction}

The theory of geometric flows on Riemannian manifolds has been widely studied over the last few decades, particularly mean curvature flow in Euclidean spaces. Following \cite{LimaJPSoliton}, an immersion of a smooth manifold $M^n$ into a Riemannian ambient space $\tilde{M}^{n+1}$ evolves under the mean curvature flow (MCF, for short) if it is the initial condition of a smooth one-parameter family of immersions evolving by their mean curvature vector field. More generally, extrinsic geometric flows are evolution equations describing hypersurfaces of a Riemannian manifold evolving in the normal direction with velocity determined by the corresponding extrinsic curvature.

The simplest solutions to the mean curvature flow are minimal surfaces, which are solutions in which the flow is constant or static. Another special class of solutions is that of solitons, also known as self-similar solutions. Self-similar solutions have played an important role in the development of the theory of the MCF, for example, in Euclidean space, where they serve as comparison solutions to investigate the formation of singularities. 

We investigate minimal surfaces and solitons for the MCF, considering a particular class of surfaces called translation surfaces. Following \cite{LopezHasanis}, the origin of such surfaces is in the classical text of Darboux \cite{Darboux}  where they are presented and later known as Darboux surfaces. Such surfaces are defined as the movement of a curve by a one-parameter family of rigid motions of $\mathbb{R}^3$.
More precisely, a surface $S \subset \mathbb{R}^3$ that can be locally written as the sum of two curves is called a translation surface. In a general Lie group $G$ with product ($ * $) one can define a translation surface $S \subset G$ as a surface in $G$  that can be locally written as the product $\Psi(s,t) = \alpha(s) * \beta(t)$ of two curves $\alpha: I \subset \mathbb{R} \to G$ and $\beta: J \subset \mathbb{R} \to G$. These curves $\alpha$ and $\beta$ are called the generating curves of $S$. 

This is inspired by previous work that has appeared on the study of translation surfaces with constant mean curvature (CMC for short) such as \cite{ferreiraSantosS3,LopezHasanis,Munteanu,Lopez,LopezMunteanu,LopezPerdomo,MoruzMunteanu,Yoon}. These prior works primarily focused on 3-dimensional Thurston geometries, which are also Lie groups. A central aim of this work is to specifically investigate translation surfaces in a 3-dimensional non-Euclidean space form, namely the hyperbolic space $\mathbb{H}^3$.

In section \ref{sec_prelim}, we outline the Lie group structure of $\mathbb{H}^3$ that we will use. This structure is defined by considering the upper half-space model and the group of similarities of $\mathbb{R}^2$. Once we have established this framework, we present two types of translation surfaces to be considered. In both cases, the surfaces are parametrized so that one of the generating curves lies within a horosphere and the other in a totally geodesic plane. We first consider the case where the generating curves are local graphs. We then apply this structure to obtain results for arbitrary curves that lie within the specified subsets. It is important to note that our definition of translation surfaces is significantly more general than the one considered in \cite{Lopez}, where translation surfaces in $\mathbb{H}^3$ within the half-space model were given as translation surfaces in $\mathbb{R}^3$, again for particular classes of curves.

Our results are categorized into three classes of surfaces in $\mathbb{H}^3$: minimal surfaces, hyperbolic translators (see \cite{LimaJPSoliton}) and conformal solitons to mean curvature flow (see \cite{MariOliveira}). The results serve as both classification and rigidity results, since they provide conclusive information about the type of surface obtained when the generating curves are as described above, and these conclusions align with known results in the existing literature. Specifically, we prove that minimal surfaces are totally geodesic planes or a class of minimal surfaces in $\mathbb{H}^3$, which were described in \cite{LopezSingular} as singular minimal surfaces in $\mathbb{R}^3$. Such surfaces are weighted minimal surfaces in $\mathbb{R}^3$ and, when viewed as subsets $\mathbb{H}^3$ in the upper half-space model, they are minimal surfaces. For the hyperbolic translators and conformal solitons to the mean curvature flow, we show that when translation surfaces are generated by curves contained in a totally geodesic plane and in a horosphere, the resulting surfaces are either the totally geodesic planes or the so-called grim-reapers cylinders in each corresponding context, i.e., the translators as given in \cite{LimaJPSoliton} and the conformal solitons in \cite{MariOliveira}. 

The main results of this work are summarized as follows: 
\begin{theorem}\label{Theo_general}
    Let $\alpha: I \subset \mathbb{R} \to \mathbb{H}^3$ be a curve contained in a horosphere $\mathcal{H}_a = \{ (x,y,a) \in \mathbb{H}^3 :  a \textnormal{ is constant}  \}$ and $\beta: J \subset \mathbb{R} \to \mathbb{H}^3$ is a plane curve contained in a totally geodesic hyperbolic plane. Then the translation surface $X: I \times J \to \mathbb{H} ^3 $, $X(s,t) = \alpha(s) * \beta(t) $, up to rigid motions,  is
    \begin{enumerate}
        \item minimal if and only if it is contained in a totally geodesic hyperbolic plane or in a minimal translation cylinder.
        \item a \textit{hyperbolic translator} if and only if it is contained in a horosphere, a totally geodesic hyperbolic plane, or in a hyperbolic grim reaper cylinder. 
        \item a \textit{conformal soliton} if and only if it is contained in a totally geodesic hyperbolic plane or in a conformal grim-reaper cylinder.
    \end{enumerate}
\end{theorem}

\section{Preliminaries} \label{sec_prelim}

Following \cite{LimaJPSoliton}, let $f : M^n \to \tilde{M}^{n+1}$ be an immersion of a smooth manifold $M^n$ into a Riemannian ambient space $\tilde{M}^{n+1}$. It evolves under the mean curvature flow (MCF for short) if there exists a smooth one-parameter family of immersions $F : M \times I \to \tilde{M}$, with $I = [0, T)$ and $ F_t(\cdot,0) = f( \cdot)$, is an immersion and satisfies the evolution equation
$$         \dfrac{\partial F}{\partial t} (p,t) = \overset{\rightarrow}{H}(p,t), \ \ \ (p,t) \in M \times I. $$
where $\overset{\rightarrow}{H}(p,t)$ is the mean curvature vector field of the immersion $F_t : M \to \tilde{M}$, where $F_t(p) := F(p, t)$ for a fixed $t \in I$. Initially, we approach the static solution to the MCF, namely, minimal surfaces. Further, we consider a special class of solutions to the MCF, called solitons, also known as self-similar solutions.

Throughout this work, consider in $\mathbb{R}^3$ the half-space $\mathbb{R}^3_+ = \{ (x,y,z) \ | \ z >  0 \} $,  where for each $p \in \mathbb{R}^3_+ $, we set the inner product 
$$ \langle \cdot , \cdot \rangle_{H}  = \frac{1}{z^2} \langle \cdot , \cdot\rangle, $$
where $\langle \cdot , \cdot\rangle$ is the  inner product in $\mathbb{R}^3$. The product $ \langle \cdot , \cdot \rangle_{H}  $ gives a Riemannian metric in $\mathbb{R}^3_+$ and the Riemannian manifold $(\mathbb{R}^3_+ , \langle \cdot , \cdot \rangle_{H})$ is called hyperbolic space, denoted by $\mathbb{H}^3$. This model will be called \textit{upper half-space model} for $\mathbb{H}^3$. Whenever we refer to the space $\mathbb{H}^3$, we are using this model.

\subsection{Conformal metrics and the corresponding mean curvatures } 

We say that two Riemannian metrics $g$ and $\tilde{g}$,  defined in a smooth manifold $M$, are locally conformally equivalent if there exists a function $\varphi: \Omega \subset M \to  \mathbb{R}$ such that  $\tilde{g} = e^{2\varphi} g$. In what follows, let $\varphi : \Omega \subset \mathbb{R}^3 \to \mathbb{R}$ be a smooth function, and $\Omega$ an open set of $\mathbb{R}^3$. Thus, let $g = \langle \cdot , \cdot \rangle $ be the usual inner product in $\mathbb{R}^3$, $S \subset \mathbb{R}^3$, $\tilde{S} \subset \mathbb{R}^3_+$, $p=(X_1,X_2,X_3) \in S$. Setting  $\varphi = -\ln z$ such that $e^{2\varphi} = 1/z^2$  gives 
\begin{equation}
    \langle \cdot,\cdot \rangle_H = e^{2\varphi} \langle \cdot , \cdot \rangle = \frac{1}{X_3^2} \langle \cdot , \cdot \rangle . \label{eq:conf-metric}
\end{equation}

Let $N(p)$ be the unit normal vector field to $S$ in $p = (X_1,X_2,X_3) \in S \subset \mathbb{R}^3$. Thus, the unit normal field $\tilde{N}(p)$ to $\tilde{S}$ in $ p \in \tilde{S} \subset \mathbb{H}^3$ will be given by  
\begin{equation}
    \tilde{N}(p) = \frac{1}{e^\varphi} N(p) = X_3 N(p). \label{eq:conf-normal}
\end{equation}
Now, let $\nabla$ and $\tilde{\nabla}$ be the Riemannian connections compatible with the metrics $g$ and  $\tilde{g}$, respectively. For two fields $X,Y$ in $T\mathbb{R}^3$ we have 
\begin{equation}
    \tilde{\nabla}_X Y = \nabla_X Y + \mathcal{S}_{\varphi}(X,Y), \label{eq:conf-connection}
\end{equation}
where $ \mathcal{S}_{\varphi} (X,Y) = d \varphi(X) Y + d \varphi(Y) X - \langle X, Y \rangle \nabla \varphi$, and $\nabla \varphi$ is the gradient of $\varphi$ (see \cite[Chapter II]{book-carmo}).

Such considerations allow us to obtain the principal curvatures and, consequently, the mean curvature for a surface $S$ in $\mathbb{H}^3$ by means of its principal curvatures when seen as a surface of $\mathbb{R}^3.$ Recall that, for a regular surface $S \subset \mathbb{R}^3$, the principal directions in $T_p S$ are given by orthonormal vectors $\{e_1, e_2\}$ such that $d N_p(e_i) = - k_i e_i $. Hence equations \eqref{eq:conf-metric}, \eqref{eq:conf-normal} and \eqref{eq:conf-connection} imply that $ \{ \tilde{ e_1}, \tilde{ e_2}  \} = \{ e^{-\varphi} e_1, e^{-\varphi} e_2  \}$ are principal directions for $\tilde{S} \subset \mathbb{R}_+^3$ and the principal curvatures are given by $ \tilde{k}_i =  e^{- \varphi} \left[ k_i - d \varphi ( N) \right] $. It follows that the mean curvature in $\mathbb{H}^3$ is given by 
$$\tilde{H} = \frac{\tilde{k_1} + \tilde{k_2} }{2}  = e^{-\varphi}(H - d\varphi(N)) . $$
where $d\varphi_p(w) = \langle \nabla \varphi(p), w \rangle $, for all $w \in T_p S$, and $\nabla \varphi (p) = (0,0,- 1/X_3) $. Thus,
\begin{equation}\label{eqcurvaturerelation}
     \tilde{H} = X_3 H + N_3 .
\end{equation}
Hence, a surface in $\mathbb{H}^3$ is minimal if and only if it satisfies 
\begin{equation}\label{eqminimalhyperbolic}
    X_3H + N_3 = 0.
\end{equation}

\subsection{Hyperbolic translators to the MCF in  \texorpdfstring{$\mathbb{H}^3$}{TEXT} } According to \cite{LimaJPSoliton}, these solutions to the MCF are characterized by being generated by the Killing field defined by a one-parameter subgroup of isometries of the ambient manifold. A soliton solution to $MCF$ is given by $F(p, t) = \Gamma_t(f(p))$ where $\Gamma_t$ is a subgroup of one-parameter of the group of isometries of $\tilde{M}$. Denoting with $\xi \in T\tilde{M}$ the corresponding Killing vector field, it is well known (see \cite{Hungerbuhler}) that $f : M^n \to \tilde{M}^{n+1}$ evolves under a soliton solution to MCF corresponding to $\Gamma_t$ if and only if the mean curvature $H$ of $f$ and its unit normal $N$ satisfy, up to tangential diffeomorphisms,
$$    H = \langle N, \xi \rangle. $$
When these isometries are translations along a geodesic, we call the corresponding self-similar solutions translating solitons to the given flow, and the initial hypersurfaces are known as translators. A main
feature of translators in the Euclidean space $\mathbb{R}^n$ is that they appear as type II singularities of certain compact solutions to mean curvature flow (cf. \cite[Theorem~4.1]{HuiskenSinestrari}). Three of the best-known examples of translators in $\mathbb{R}^n$ are the grim reaper cylinder, and the rotational examples known as the translating paraboloid or bowl soliton (see \cite{AltschulerWu}) and the translating catenoids (see \cite{ClutterbuckSchulze}).

Following the exposition in \cite{LimaJPSoliton}, in the context of hyperbolic space $\mathbb{H}^3$, let $\mathcal{G}:=\{ \Gamma_t \ : \  t \in \mathbb{R}\}$  be the one-parameter subgroup of  the isometry group of $ \mathbb{H}^3$, where $\mathcal{G}$ comprises the hyperbolic translations along the $x_3$-axis defined by $ \Gamma_t(p) = e^t p$, $p \in \mathbb{H}^3. $ A solution to the (MCF) given by $F(p,t) = e^t p$ will be called a \textit{translating soliton} with respect to the group $\mathcal{G}$. This name reflects the self-similar evolution under the flow generated by the group action. To describe the associated Killing field, we take a slight abuse of notation by identifying $\mathbb{H}^3 $ with $T_p \mathbb{H}^3$. With this identification, the Killing field associated with the group $\mathcal{G}$ is given by  $\xi(p) = p \in \mathbb{H}^3$. The surface $\tilde{S}$ with unit normal $\tilde{N}$ is the initial condition of a $\mathcal{G}$-soliton generated by $\xi$ if and only if the equality 
\begin{equation}\label{eq_cond_soliton_theory}  
    \tilde{H} = \langle \xi , \tilde{N}  \rangle .
\end{equation}
holds everywhere in $\tilde{S}$ (see \cite{Hungerbuhler}). In this case, $\tilde{S}$ is  called \emph{hyperbolic translator}.
Thus, it follows from equation \eqref{eq_cond_soliton_theory} that a surface $\tilde{S} \subset \mathbb{H}^3$ is a hyperbolic translator if and only
$$   \tilde{H} = \frac{1}{X_3} \langle p,  N (p) \rangle . $$
By equation \eqref{eqcurvaturerelation}, this condition is equivalent to  
\begin{equation}\label{eqconditionsoliton}
    X_3^2 H = X_1 N_1 + X_2 N_2. 
\end{equation}

\subsection{Conformal solitons to the MCF in  \texorpdfstring{$\mathbb{H}^3$}{TEXT} } Another particular class of solutions to the MCF that we will approach here is the \textit{conformal solitons}. In this case, we can proceed as above, but considering $\xi$ as a conformal vector field, instead of a Killing field. Formally, a vector field $\xi$ is conformal if its Lie derivative with respect to the metric $g$, denoted by $\mathcal{L}_\xi g$, satisfies $\mathcal{L}_\xi g = 2 \lambda g$, where $\lambda $ is a smooth function called the potential function. In particular, if $\lambda$ is a constant, the vector field is called homothetic, and if $\lambda$ is identically zero, it is a Killing vector field (see  \cite{AliasLira,ArezzoSun,GiulioColombo,Smoczyk}). 
Conformal solitons in the hyperbolic space $\mathbb{H}^n$ were firstly considered by \cite{MariOliveira}. In this case, if $\xi$ is a conformal vector field in $\mathbb{H}^n$, a conformal soliton satisfies the following equation
\begin{equation}\label{eq_def_solitonconf_general}
    \tilde{H} = \langle  \xi, \tilde{N} \rangle .
\end{equation}
In \cite{MariOliveira}, the authors focus their approach on the conformal field  $\xi(p) = - e_n$, and that is the vector field we will consider here. Hence, for the conformal field  $\xi(p) = - e_3$, for all $p \in \mathbb{H}^3$, a surface $\tilde{S}$ that satisfies $\tilde{H} = - \langle e_3, \tilde{N}(p)  \rangle$. Thus, equation \eqref{eq_def_solitonconf_general} becomes   
$$ \tilde{H} = -\frac{1}{X_3} \langle e_3 , N(p) \rangle = -\frac{N_3}{X_3}. $$
Using equation \eqref{eqcurvaturerelation} we obtain  
\begin{equation}\label{eqconditionconformalsoliton}
    X_3^2 H =  - (X_3 + 1) N_3. 
\end{equation}
 
\subsection{Translation surfaces in  \texorpdfstring{$\mathbb{H}^3$}{TEXT} }

We recall that in the Euclidean space $\mathbb{R}^3$ a surface $S$ is called a translation surface if for two curves $\alpha : I \subset \mathbb{R} \to \mathbb{R}^3 $ and  $\beta : J \subset \mathbb{R} \to \mathbb{R}^3 $, it is given by $ X:  I \times J  \to  \mathbb{R}^3 $, $ (s,t) \mapsto  \alpha(s) + \beta(t) $. It is well known that a minimal surface is a surface where the mean curvature $H$ vanishes everywhere. These minimal translation surfaces in $\mathbb{R}^3$ were classified in \cite{LopezHasanis}. 
In particular, as a motivation for this paper, we refer \cite{LopezHasanis,LopezPerdomo,Lopez,Yoon}

We now also recall that the hyperbolic space $\mathbb{H}^n$ for $n \geq 2 $ is a non-commutative metric Lie group. Following \cite{MeeksPerez}, for the upper half-space model for $\mathbb{H}^n$, it can be seen as the group of similarities of $\mathbb{R}^{n-1}$ by means of the isomorphism 
$$\begin{array}{ccccc}
      \psi_{(x,z)} : & \mathbb{R}^{n-1} & \to & \mathbb{R}^{n-1} \\
        & y & \mapsto &   z y + x 
\end{array}$$
In general, the group operation $*$ for a semidirect product of the form $\mathbb{R}^2 \rtimes_\varphi \mathbb{R}$ is given by 
$$ (x_1,z_1) * (x_2,z_2) = ( x_1 + \varphi_{z_1} (x_2), z_1 + z_2 ), $$
where $\varphi$ is given by the exponential of some matrix $A \in \mathcal{M}_2(\mathbb{R})$, that is, $\varphi_z(x) = e^{z A } x $, and we denote the corresponding group by $\mathbb{R}^2 \rtimes_A \mathbb{R}$. If $A$ is the identity matrix $I_2 \in \mathcal{M}_2(\mathbb{R})$, then $e^{z A } = e^z I_2$ and we recover the group of similarities of $\mathbb{R}^2$. Moreover, the map 
$$ (x,y,z) \in \mathbb{R}^2 \rtimes_{I_2} \mathbb{R} \overset{\Phi}{\mapsto} (x,y,e^z) \in \mathbb{R}^3_+ . $$
gives an isomorphism between $\mathbb{R}^2 \rtimes_{I_2} \mathbb{R}$ and the upper halfspace model $\mathbb{R}^3_+$
for $ \mathbb{H}^3$ with the group structure given previously. Thus, we have 
$$ (x_1,y_1,w_1) * (x_2, y_2 , w_2) \overset{\Phi}{\mapsto}
(x_1,y_1,e^{w_1}) * (x_2, y_2 , e^{w_2}).$$
Hence, for any $(x_1,y_1,z_1), (x_2,y_2,z_2) \in \mathbb{H}^3$ we can define  the Lie product
$$(x_1,y_1,z_1) * (x_2,y_2,z_2) = (z_1 x_2 + x_1  , z_1 y_2 + y_1 , z_1z_2). $$
In this set, considering two regular curves $\alpha : I \subset \mathbb{R} \to \mathbb{H}^3 $ and  $\beta :  J \subset \mathbb{R} \to \mathbb{H}^3$, we call $\alpha * \beta : I \times J \to \mathbb{H}^3$ a translation surface.

Now, let $\alpha : I \subset \mathbb{R} \to \mathbb{H}^3$ and $\beta : J \subset \mathbb{R} \to \mathbb{H}^3$ be two regular curves and suppose that $\alpha$ is contained in the horosphere $\mathcal{H}_h = \{ (x,y,z) \in \mathbb{R}^3_+ : z = h  \}$, more precisely, in the horosphere $\mathcal{H}_1$. Then such a curve can be locally parameterized as   
 $$\alpha(s)  =  (\alpha_1(s),\alpha_2(s),1). $$
Consider also $\beta$ that is contained in the totally geodesic hyperbolic plane that contains the $e_3$ axis, $P_{e_1,e_3} = \{ v \in \mathbb{H}^3 : v = (v_1,0,v_3)  \}$. Thus, such a curve can be parameterized as follows 
$$ \beta(t) =  (\beta_1(t) , 0,  \beta_2(t)). $$
Now, let $A$ be a linear map of the form 
$$ A = \begin{bmatrix}
\cos\theta & - \sin(\theta) & 0 \\
\sin\theta & \cos\theta & 0 \\
0 & 0 & 1 
\end{bmatrix}, \ \ \ \theta \in \mathbb{R}. $$
 This map is a rotation, in the usual sense, around the $e_3$ axis in $\mathbb{R}^3$ and is also well known to be an isometry of $\mathbb{H}^3$. It is easy to see that 
$$ A (\alpha(s) * \beta(t)) =  (A \alpha(s)) * (A \beta(t)) .$$
Since $A(\alpha(s)) = \gamma(s)$ is still a curve in $\mathcal{H}_1$, then the surface $X(s,t) = \alpha(s) * \beta(t)$ is merely a rigid motion of the surface $Y(s,t) = \gamma(s) * (A \beta(t))$. 

Hence, for the purpose of this work, our local analysis of a translation surface can be simplified as follows: let $\alpha : I \subset \mathbb{R} \to \mathbb{H}^3$ and $\beta : J \subset \mathbb{R} \to \mathbb{H}^3$ be two curves in $\mathbb{H}^3$, and suppose that $\alpha$ is contained in the horosphere $\mathcal{H}_1$. Since $\mathcal{H}_1$ is a plane that is orthogonal to the $e_3$ axis, up to parabolic translations, such a curve can be locally parameterized as a graph of the form $\alpha(s)  =  (s,f(s),1)$  or a line of the form $\alpha(s)  =  (a,s,1)$, with $a \in \mathbb{R}$. Consider also a curve $\beta$ contained in some totally geodesic hyperbolic plane. Again, since it is a plane, then, up to parabolic translations, such curves can be parameterized as $ \beta(t) =  ( 0, t, g(t)) $ or a line $ \beta(t) =  ( 0, b , t )$, with $ b \in \mathbb{R}$.

Hence, we can reduce the process to 4 possible cases: \\
If $\alpha(s)  =  (s,f(s),1)$ and  $ \beta(t) =  (0 , t ,g(t)) $, we have 
\begin{equation}\tag{First kind}
    \alpha(s) * \beta(t)  = (s ,  t + f(s), g(t)) 
\end{equation}
If $\alpha(s)  =  (s,f(s),1)$  and $ \beta(t) = ( 0 , b , t ) $,  we have 
\begin{equation}\tag{Second kind}
    \alpha(s) * \beta(t) = (s , f(s) + b , t)    
\end{equation}
If $\alpha(s)  =  (a, s , 1)$  and $ \beta(t) =  (0 , t ,g(t)) $,  we have $\alpha(s) * \beta(t) = (a ,  t + s , g(t)),$ and if  $\alpha(s)  =  (a,s,1)$  and $ \beta(t) = ( 0 , b , t ) $,  we have $ \alpha(s) * \beta(t) = (a ,  s + b , t )$. Both surfaces are clearly contained in totally geodesic hyperbolic planes. Once the local analysis is established, we proceed with the global analysis to obtain Theorem \ref{Theo_general}.

\section{Main results}

Before we proceed, we establish the notation $ f'(s) = \dfrac{d}{ds} f(s), $ that will be used throughout this work. These derivatives will often appear in the simplified form $f'$,  unless otherwise stated. The purpose of this simplification is to make the notation more concise in the computations.

Under these conditions, we will prove the following results in the last section of this work.

\begin{theorem}\label{theominfirst} 
Let $S \subset \mathbb{H}^3$ be a translation surface of the first or second kind. Then $S$ is minimal if and only if it is either contained in a totally geodesic plane in the second kind case, or, for constants \( c, d, m \in \mathbb{R} \),  it can be locally parameterized as 
$$X(s,t) = (s , p(s) + t, q(t)),$$
where $p(s) = c s + d $ and $q(t)$ is the solution of the ordinary differential equation 
\begin{equation}\label{eq_sol_beta_minimal}
     (q')^2(t) = \frac{m}{q^4(t)} - \frac{1}{c^2 + 1 } , \ \ \ m > 0,
\end{equation}
in the first kind case.
\end{theorem}

\begin{remark}
     Following \cite{LopezSingular}, complete solutions of equation \eqref{eq_sol_beta_minimal} yield surfaces referred to as "$\rho$-catenary translation cylinders" in $\mathbb{R}^3$. The case where $\rho = -2$ corresponds to a minimal surface in $\mathbb{H}^3$, when viewed in the upper halfspace model.
\end{remark}

For the hyperbolic translator case, we have the following. 

\begin{theorem}\label{theoSolitonfirst}

Let \( S \subset \mathbb{H}^3 \) be a translation surface of the first or second kind. Then \( S \) is a hyperbolic translator if and only if it is contained in a totally geodesic hyperbolic plane in the second kind case, or, for the first kind case, it is contained either in a horosphere, or, for constants \( c, d \in \mathbb{R} \), it can locally be parametrized as 
$$X(s,t) = (s , p(s) + t, q(t)),$$
where $p(s) = c s + d $ and $q(t)$ is the solution of the ordinary differential equation
\begin{equation}\label{eq_soliton_JP_Main theorem}
    q''(t)=  - 2q'(t) \left[(q')^2(t) + \frac{1}{c^2 + 1} \right] \frac{d + t}{q^2(t)}.
\end{equation}
\end{theorem}

\begin{remark}
Following \cite[Theorems~3.20 and 3.23]{LimaJPSoliton}, the so-called hyperbolic grim reaper surfaces are described as a one-parameter family of non-congruent, complete translators, which are horizontal parabolic cylinders, whose directrix curves are given as solutions of equation \eqref{eq_soliton_JP_Main theorem}.
\end{remark}

For the conformal soliton case, we have the following. 

\begin{theorem}\label{theconfSolitonfirst}
Let \( S \subset \mathbb{H}^3 \) be a translation surface of the first or second kind. Then \( S \) is a conformal soliton if and only if it is contained in a totally geodesic hyperbolic plane in the second kind case, or, for constants \( c, d, m \in \mathbb{R} \), it can be parameterized as
$$ X(s,t) = (s,\ t + p(s),\ q(t)),$$
where $ p(s) = c s + d $, with $ c \ne 0 $, and $ q(t) $ is either a non-negative constant or a solution of the ODE
\begin{equation}\label{eq_ODE_transsoliton_general theroem}
     (q')^2(t) = \frac{m e^{4/q(t)}}{q^4(t)}-  \frac{1}{1 + c^2}, \ \ \ m > 0.
\end{equation}

\end{theorem}

\begin{remark}
The conformal grim reaper cylinders are described in \cite[Theorem~C]{MariOliveira} as conformal solitons with respect to the conformal vector field \( e_n \) in the upper half-space model of \( \mathbb{H}^{n+1} \), are analogous to those described in Theorem~\ref{theconfSolitonfirst}. The solitons presented there are graphs of the form
$$ \Gamma = \{ (x, u(x)) \in \mathbb{R}^{n+1}_+ = \mathbb{R}^n \times (0, +\infty) : x \in \Omega \subset \mathbb{R}^n \},$$
with $ u \in C^{\infty}(\Omega) $.
\end{remark}

\begin{figure}[H]
\centering
\includegraphics[width=0.3\linewidth]{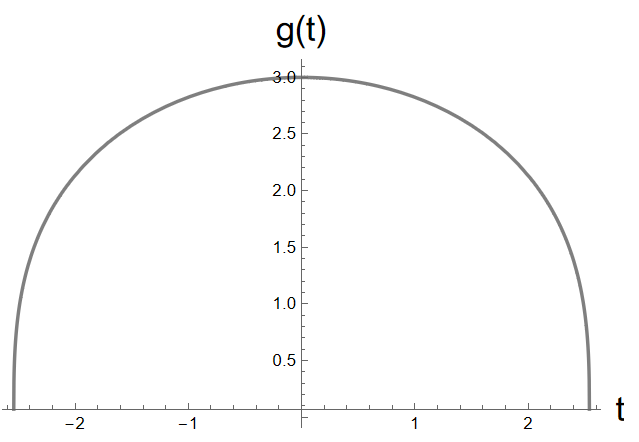}
\includegraphics[width=0.45\linewidth]{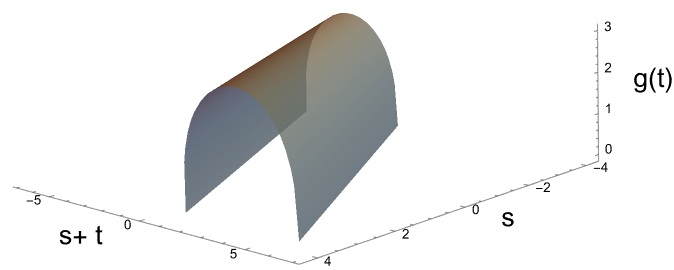}
\caption{\label{fig:Theo_min_1_func} ODE numerical solution and surface of Theorem \ref{theominfirst}, with $y_0 = 3$, $c = 1$ and $d = 0$}
\end{figure}

\begin{figure}[H]
\centering
\includegraphics[width=0.5\linewidth]{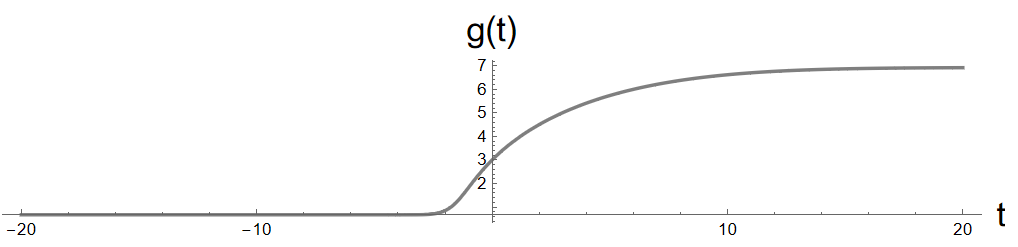}
\includegraphics[width=0.5\linewidth]{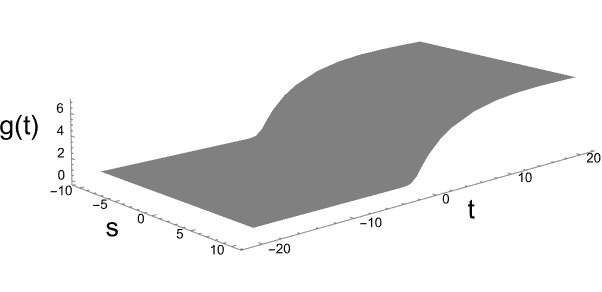}
\caption{\label{fig:Theo_sol_2}  ODE numerical solution and Surface of Theorem \ref{theoSolitonfirst}, with $y(0) = 3$, $y'(0) = 1$, $a = 1$ and $k = 1/2$.}
\end{figure}

\begin{figure}[H]
\centering
\includegraphics[width=0.3\linewidth]{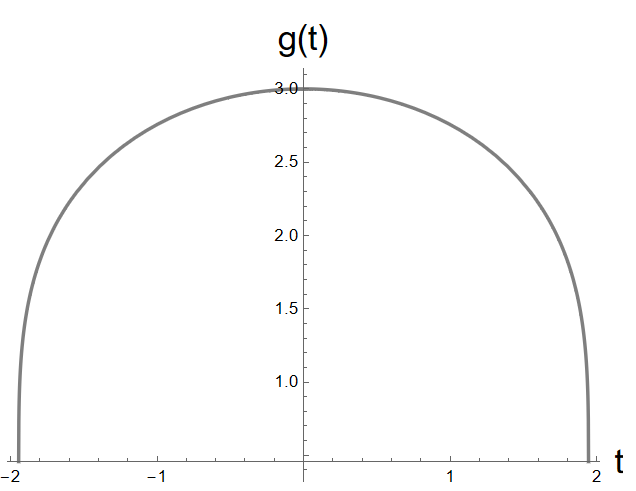}
\includegraphics[width=0.35\linewidth]{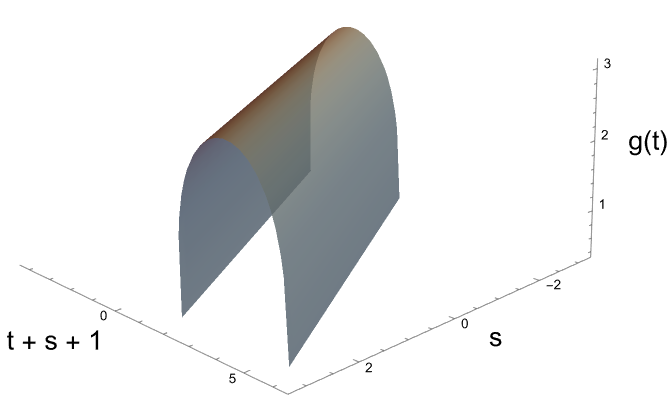}
\caption{\label{fig:Theo_conf_sol_2}  Ode solution and Surface of Theorem \ref{theconfSolitonfirst}, with $y_0 = 3$, $c = 1$ and $d = 0$.}
\end{figure}

\begin{remark}
    At first glance, minimal translation cylinders and conformal grim-reaper cylinders may appear similar; however, this is not the case. In fact, it follows from equations \eqref{eqminimalhyperbolic} and \eqref{eqconditionconformalsoliton} that a surface $S$ is a conformal soliton and a minimal surface if and only if is a totally geodesic plane.
\end{remark}

\section{Geometry of translation surfaces in \texorpdfstring{$\mathbb{H}^3$}{TEXT}} To establish the corresponding equations for each class of surfaces we are considering, we first compute $H$ and $N$ in the Euclidean space and then use equation \eqref{eqconditionsoliton} to obtain $\tilde{H}$. Thus, for a surface of the form 
$$\begin{array}{cccccc}
     X : & I \times J & \to & \mathbb{R}^3 \\ & (s,t) & \mapsto & \alpha(s) * \beta(t).  
\end{array}$$
Let $X_s$ and $X_t$ be the derivatives of $X$ with respect to $s$ and $t$, respectively. We compute the first and second fundamental forms coefficients  in the standard way 
$$ \begin{array}{rclcrclcrcl}
     E & = & \langle X_s , X_s \rangle & , &  
     G & = & \langle X_t , X_t \rangle & , &
     F & = & \langle X_s , X_t \rangle, \\
     l & = & \langle X_{ss} , N \rangle & , &  m & = & \langle X_{tt} , N \rangle, & , & n & = & \langle X_{st} , N \rangle,
\end{array}$$
where $N(s,t)$ is the normal Gauss map over $X(s,t)$. Then, we use the classical definition to obtain the mean curvature, that is 
$$H = \frac{lG - 2nF + Em}{2(EG-F^2)}. $$

Consider now two curves in $\mathbb{H}^3$ locally parameterized as  $ \alpha(s) = (s,f(s),1)$ and $ \beta(t) = (0, t,g(t))$. For such curves, we have
$$ X(s,t) = \alpha(s) * \beta(t) = (s ,  f(s) + t, g(t)).  $$
We compute 
$$\begin{array}{ccccc}
     X_s = (1,f',0), &  X_t = (0,1,g'), & X_{ss} = (0,f'',0),  & X_{st} = (0,0,0) , &  X_{tt} = (0,0,g''). 
\end{array}$$
Hence the unit normal $N$ is given by  
\begin{equation}\label{eq_normal_firstkind}
    N = \frac{X_s \times X_t}{|X_s \times X_t|} = \frac{( f'g', - g' , 1 )}{\sqrt{(g')^2((f')^2 + 1 ) + 1}},
\end{equation}
and the coefficients of the first fundamental form are
$$  E =  1 + (f')^2. \ \ \ G =  1 + (g')^2,  \ \ \ F = f'. $$
The coefficients of the second fundamental form are given by 
$$  l = \dfrac{-f''g'}{\sqrt{(g')^2((f')^2 + 1 ) + 1 }}, \ \ \ m =  \dfrac{g''}{\sqrt{(g')^2((f')^2 + 1 ) + 1}},  \ \ \ n =   0 . $$
Therefore,
\begin{equation}\label{eq_meancurvature_firstcase_simplified}
    H =  \frac{-f''g'(1+(g')^2) + g''(1+(f')^2)}{2[(g')^2((f')^2 + 1 ) + 1]^{3/2}}.
\end{equation}

Now consider the curves $ \alpha(s) = (s,f(s),1)$ and $  \beta(t) = (0,b,t)$, $ b \in \mathbb{R}$. For these curves, we have
$$ X(s,t) = \alpha(s) * \beta(t)  = (s , f(s) + b  , t) . $$
We compute 
$$
\begin{array}{ccccc}
     X_s = (1,f',0), & X_t = (0, 0 ,1), & X_{ss} = (0,f'',0) , & X_{st} = (0,0,0), &  X_{tt} = (0, 0,0) .     
\end{array}
$$
In this case, we have 
\begin{equation}\label{eq_normal_secondkind}
N = \frac{X_s \times X_t}{|X_s \times X_t|} = \frac{( f', - 1 , 0 )}{\sqrt{(f')^2 + 1}}
\end{equation}
The coefficients of the first and second fundamental forms are given respectively by
$$ E = 1 + (f')^2, \ \ \ G =  1 , \ \ \ F =  0, \ \ \textnormal{and} \ \ \ l =  \frac{-f''}{\sqrt{(f')^2 + 1}}, \ \ \ m = 0, \ \ \ n  =  0 .$$
Therefore,
\begin{equation}\label{eq_meancurvature_secondcase_simplified}
    H =  \frac{-f''}{2[(f')^2 + 1 ]^{3/2}}.
\end{equation}
\subsection{Auxiliary results}

We now present the following lemmas, that describe the behavior of each directrix curve of the cylinders given in Theorem \ref{theoSolitonfirst}. The first lemma concerns minimal cylinders:
\begin{lemma}[{\cite[Theorem~3]{LopezSingular}}]\label{teoresultgfirstkind}
Let $g$ be a solution of 
$$    g''g(1+c^2) = - 2 [(g')^2(1 + c^2)  + 1] ,$$
with the initial conditions $ g(0)= y_0 > 0$, $g'(0) = 0$.
Then $g$ is defined in an interval $(-r,r)$, is concave and symmetrical with relation to the $z$ axis, with a point of maximum in $t = 0$. Also, $\lim\limits_{t \to \pm r} g(t) = 0 $ and $\lim\limits_{t \to \pm r} g'(t) = \pm \infty$.
\end{lemma}

For the directrix curve of the hyperbolic grim reaper cylinder, we have the following:
\begin{lemma}[{\cite[Lemma~3.19]{LimaJPSoliton}}]\label{lemma_soliton_JP}
    Given $\lambda \geq 0$ and $k > 0$, the initial value problem
    $$ g'' = - g'(k + (g')^2)\dfrac{2t}{g^2}, $$
    with initial conditions $g(0) = 1$, $g'(0) = \lambda$,
    has a unique smooth solution $g : \mathbb{R} \to (0, + \infty)$ which has the following properties:
    \begin{enumerate}
        \item $g$ is constant if $\lambda = 0$.
        \item $g$ is increasing, convex in $(-\infty, 0)$, and concave in $(0, + \infty )$ if $\lambda > 0$.
        \item $g$ is bounded from above by a positive constant.
        \item $g$ is bounded from below by a positive constant
    \end{enumerate}
\end{lemma}

Consider now the following well-known result 
\begin{theorem}[Theorem 10.12 of \cite{Doering}]\label{teolimitedo}
Let $E$ be an open set of $\mathbb{R}^n$ that contains $x_0$, $f \in C^1(E)$, and $(\alpha,\beta)$ a maximal interval of existence of the following I.V.P.
 $$   \dot{x} = f(x), \ \ \ x(0) = x_0 .$$
If $\beta < \infty$ (similarly $\alpha> -\infty$), then for every compact set $K \subset E$ there exists $t \in (\alpha, \beta)$ such that $x(t) \not\in K$.
\end{theorem}

The following lemma addresses an ODE similar to the one presented in the proof \cite[Theorem~C]{MariOliveira}, which describes the directrix curve of the conformal grim reaper cylinder. Since the structure of the proof and the assumptions here are different, we present the lemma with its proof to make the exposition self-contained.
\begin{lemma}\label{theo_anal_theo_conf_sol_2}
Let $g$ be a solution of 
\begin{equation}\label{eq_conf_teo_sol_1_anal_1}
     g''  + 2\frac{(g+1)}{g^2}(g')^2  + 2\frac{(g+1)}{g^2(1+a^2)}  = 0 ,
\end{equation}
with initial conditions 
\begin{equation}\label{eqcondinic_conf_sol}
g(0)= y_0 > 0, \ \ \ \ g'(0) = 0 .
\end{equation}Then $g$ is defined in an interval $(-r,r)$, is concave and symmetric with respect to $z$  axis, with a maximum in $t = 0 $. Also, 
$\lim\limits_{t \to \pm r} g(t) = 0 $ and $\lim\limits_{t \to \pm r} g'(t) = \pm \infty$.
\end{lemma}
\begin{proof}
We rewrite equation \eqref{eq_conf_teo_sol_1_anal_1} as  
$$ g'' (t) = - 2\frac{(g(t) + 1)}{g^2(t)} \left[ (g'(t))^2  + \frac{1}{(1+a^2)}\right] . $$

Now, since $g$ is the solution of \eqref{eq_conf_teo_sol_1_anal_1}, define $h(t) = g(-t)$. Thus, $ h'(t) = -g'(-t)$ and $h''(t) = g''(-t) $. Also, we have 
$$ h''(t) = - 2\frac{h(t) + 1}{h^2(t)} \left(h'^2(t) + \frac{1}{1 + a^2} \right). $$
In addition, $h$ satisfies the initial conditions \eqref{eqcondinic_conf_sol}. Therefore, by the theorem of existence and uniqueness of ODE solutions, we conclude that $ g(t) = g(-t) $, so $ g $ is symmetric with respect to the $ z $-axis. Hence, $ g $ is defined in an interval of the form $ (-r, r)$. Since $ g $ is a positive function and satisfies $ g''(t) < 0 $, it follows that $g$ is concave. By symmetry, $ g $ achieves a unique maximum at $ t = 0 $. Now, since  $g$ is decreasing in $(0,r)$, it follows from Theorem \ref{teolimitedo} that 
    $$\lim_{t \to  r} g(t) = 0  \ \ \textnormal{and} \ \  \lim_{t \to r} g'(t) = - \infty.$$
    Similarly, we have 
    $$\lim_{t \to  -r} g(t) = 0  \ \ \textnormal{and} \ \  \lim_{t \to -r} g'(t) = \infty. $$
\end{proof}

\section{Proof of the main results }

\subsection{Proof of theorem \ref{theominfirst}}
\begin{proof}
We begin by considering parameterizations of the first kind. Using equations \eqref{eqminimalhyperbolic}, \eqref{eq_normal_firstkind} and \eqref{eq_meancurvature_firstcase_simplified}, we have 
\begin{equation}\label{eqorig2}
-f''gg'(1+(g')^2) + g''g(1+(f')^2) = - 2[(g')^2((f')^2 + 1) + 1] .
\end{equation}
 
Suppose initially that  $f = c s + d $, $c, d \in \mathbb{R}$. Thus, equation \eqref{eqorig2} becomes
\begin{equation}\label{eq_part_sol_teo_min2}
    g''g(1+c^2) = - 2 [(g')^2(1 + c^2)  + 1].
\end{equation}
The solution of such an ODE is described by Lemma \ref{teoresultgfirstkind}. Observe that if $g'' \equiv 0$ then $g(t) = m t + n$ and equation \eqref{eq_part_sol_teo_min2} becomes  $ m^2(1 + c^2) + 1 = 0 $, a contradiction. On the other hand, supposing that $g'' \not \equiv 0 $, and since $g > 0$ then, setting $v(g) = (g')^2$ so that $v' = 2g''$, gives
$$ v' g = -4 (v  + m) , \ \ \ m= \frac{1}{c^2 + 1 } . $$
A first integration of the above equation yields
$$    (g')^2(t) = \frac{m} {g^{4}(t)}  - \frac{1}{c^2 + 1 }.$$
Now, suppose $f'' \not\equiv 0$ and $g'' \equiv 0$, then $g(t) = ct + d$ and equation \eqref{eqorig2} becomes
$$ f''g c (1+c^2) = 2[c^2((f')^2 + 1) + 1] . $$
Differentiating with respect to $t$ gives $c^2(1+c^2) \equiv 0$, a contradiction to the above equation.

From now on, suppose that $f'' \not\equiv 0 $ and $g'' \not\equiv 0$ .  Differentiating  equation \eqref{eqorig2} with respect to $t$ gives 
\begin{equation}\label{eq_proof_minimal_1deriv}
     -f''[gg'(1+(g')^2)]' + (1+(f')^2)[(g''g)' + 4g'g''] = 0.
\end{equation}
Since $f'' \not\equiv 0 $, we divide both sides of the equation by $f''$ and differentiate with respect to $s$ to obtain 
$$ [(g''g)' + 4 g'g''] \bigg[\frac{1+(f')^2}{f''}\bigg]' = 0. $$
If $ (g''g)' = - 4g'g'' $ then $ g''g =  - 2 (g')^2 + k_1 $. Set  $v(g) = (g')^2 $, so that $v' = 2g''$ and we have 
$$ v' = \frac{2}{g} (-2v + k_1 ) . $$
A first integration gives 
$$ (g')^2 = \frac{k_2}{g^4} + \frac{k_1}{2}, \ \ \ k_2 \ne 0 .$$
Furthermore, if $ (g''g)' = - 4g'g'' $, it follows from equation \eqref{eq_proof_minimal_1deriv} that $ [gg'(1+(g')^2)]' = 0 $. Thus,  
$$ M_1^2 =  [gg'(1+(g')^2)]^2 = \frac{\left(g^4 k_1 + 2 k_2 \right) \left(g^4 (k_1 + 2) + 2 k_2 \right)^2}{8 g^{10}},  $$
a polynomial equation in $g$ with constant coefficients, which implies that $g$ is constant, a contradiction. Thus, $(g''g)' + 4 g'g'' \neq 0$ and we must have  
$$ 1+(f')^2  = M_2  f'' , \ \ \ M_2 \ne 0 . $$
Hence, equation \eqref{eqorig2} becomes 
$$ f''[ -gg'(1+(g')^2) + g''g M + 2M(g')^2] = - 2.  $$
Since $f'' \not\equiv 0$, we divide both sides by it and differentiate with respect to $s$ to get $ N f'' = - 2$, where $N \ne 0$. Then 
$$ 1 + (f')^2  +  \frac{2M}{N}  = 0, $$ 
which implies that $f'$ is constant, a contradiction. 

Now we consider parameterizations of the second kind. Using equations \eqref{eqminimalhyperbolic}, \eqref{eq_normal_secondkind} and  \eqref{eq_meancurvature_secondcase_simplified}, we have
$$  - t f'' = 0  .$$
Since $t > 0$, we must have $f'' \equiv 0 $. This implies that the surface is contained in a totally geodesic hyperbolic plane.

\end{proof}

\subsection{Proof of  theorem \ref{theoSolitonfirst}}
\begin{proof}
We begin by considering parameterizations of the first kind. Using equations \eqref{eqconditionsoliton}, \eqref{eq_normal_firstkind}  and  \eqref{eq_meancurvature_firstcase_simplified}, we have 
\begin{equation}\label{eqsol1case}
-f''g^2g'(1+(g')^2) + g^2g''(1+(f')^2)  =  2g'[(g')^2((f')^2 + 1 ) + 1] [(sf' - f) - t ]  .
\end{equation}

Suppose initially that $g \equiv c > 0$. Thus, the Gauss map is given by $N \equiv e_3$ and the surface is contained in a horosphere. 

Now, suppose that $sf' - f = -a$, with $a \in \mathbb{R}$. Hence $f(s) = b s + a $, and equation \eqref{eqsol1case} becomes
\begin{equation}\label{eq_theo_sol_1_fconstant}
   g^2g''(b^2+ 1 )  =  - 2g'[(g')^2(b^2 + 1 ) + 1] [a + t ]. 
\end{equation}
Such an equation has at least a constant solution. Also, setting $v = a + t$ and $(b^2 +1)^{-1} = k $, gives $[g(a + t)]' = g'(v)$ and $[g(a+t)]'' = g''(v) $. Thus, equation \eqref{eq_theo_sol_1_fconstant} becomes 
$$  g''  =  - g'[(g')^2 + k] \frac{2v}{g^2}.$$

Suppose from now on that $g' \not\equiv 0 $, $f' \not\equiv 0$ and $(sf' - f)' \not\equiv 0 $. Dividing both sides by $g'$ and differentiating with respect to $t$ gives 
$$ -f''[g^2(1+(g')^2)]' + \bigg[\frac{g^2g''}{g'}\bigg]'(1+(f')^2)  = 2((g')^2)'((f')^2 +1)(sf' - f) - 2[t(g')^2]'((f')^2 +1) - 2 . $$ 
We now divide both sides by $((f')^2 +1)$ and 
differentiate with respect to $s$ and then $t$ to obtain  
\begin{equation}\label{eq_soliton_changeof notation}
    -\bigg[\frac{f''}{(f')^2 +1 }\bigg]' [g^2(1+(g')^2)]'' = 2((g')^2)''(sf' - f)'. 
\end{equation}
We set 
$$ F_1 = -\bigg[\dfrac{f''}{(f')^2 +1 }\bigg]', \ \ F_2 = (sf' - f)', \ \  G_1 = [g^2(1+(g')^2)]'', \ \ G_2 = ((g')^2)''. $$
Thus, equation \eqref{eq_soliton_changeof notation} becomes 
$$ F_1 G_1 = 2G_2F_2 . $$
Suppose initially that $G_2 \ne 0$ and $F_1 \ne 0 $. We have 
$$ \frac{F_2}{F_1}  = \frac{G_1}{2G_2} = - \overline{P} , \ \ \ \overline{P} \in \mathbb{R}.$$
Since each side relies on its own variable, we get 
\begin{equation}\label{eq_banana}
    - \overline{P} \bigg[\frac{f''}{(f')^2 +1 }\bigg]'  = - (sf' - f)'.   
\end{equation}
If $\overline{P} = 0$ we have $F_2 \equiv 0$, a contradiction. Suppose that $\overline{P} \ne 0$ and set $P = 1/\overline{P}$. After a first integration, equation  \eqref{eq_banana} becomes 
\begin{equation}\label{eqgeralsolit}
\begin{array}{rcl}
     f'' =   ((f')^2 +1)[P(sf' - f) + Q]. 
\end{array}
\end{equation}
Dividing both sides of equation \eqref{eqsol1case} by $(1+ (f')^2)$ and differentiating them with respect to $s$ 
and then $t$, yields 
$$ -P(sf' - f)'[g^2(1+(g')^2)]' =  2((g')^2)'(sf' - f)'  - 2\bigg[\frac{1}{(f')^2 + 1}\bigg]'. $$
Since $(sf' -  f )' \not\equiv 0$, we divide both sides of the previous equation by $(sf' -  f)'$ and differentiate with respect to $s$ to obtain 
$$ \bigg[\frac{1}{(f')^2 + 1}\bigg]'  = M (sf - f)' . $$
If $M = 0 $, then  $f'$ is constant, a contradiction since $F_1 \not\equiv 0 $. Thus, $M \ne 0$ and 
$$\frac{1}{P} \bigg[\frac{f''}{(f')^2 +1 }\bigg]'  = (sf' - f)'  = \frac{1}{M} \bigg[\frac{1}{(f')^2 + 1}\bigg]', $$
That is, 
\begin{equation}\label{eq_soliton_solve_error_JP}
    f''  =  \frac{P}{M} + N_1((f')^2 + 1), \ \ \ \textnormal{and} \ \ \  (sf' - f) = \frac{1+ R ((f')^2 +1 )}{M((f')^2 + 1)} . 
\end{equation}
Observe that if $N_1 = 0 $, then  $f''$ is constant, which implies that $f(s) = \frac{P}{2M} s^2 + b s + c $ and $f' = \frac{P}{M} s + b $. But, in view of \eqref{eqgeralsolit} we have
$$ \frac{P^3}{2M} s^4 + D(s) = 0 , $$
where $D(s)$ is a polynomial of degree $3$. We conclude that $P = 0$, a contradiction. Thus, $N_1 \ne 0$. 

Since $f'' \not\equiv 0 $, we differentiate the second equation of \eqref{eq_soliton_solve_error_JP} twice with respect to $s$ to get  
$$ M = - 2 f'' \left[ \frac{1}{(1+ (f')^2)^2} - 4\frac{(f')^2}{(1+ (f')^2)^3} \right] . $$
By the first equation of \eqref{eq_soliton_solve_error_JP}, we have 
$$ -M^2 f'^6 + (6 M N_1 - 3 M^2) f'^4 + (4 M N_1 + 6 P - 3 M^2) (f')^2 - M^2 - 2 M N_1 - 2 P =0 .$$ 
A polynomial equation in $f'$, which leads to $f'$ being constant, a contradiction.

It remains to evaluate the case where $F_1 \equiv 0 $ and/or $G_2 \equiv 0 $. If $F_1 \not\equiv 0 $ and $G_2 \equiv 0$, from equation \eqref{eq_soliton_changeof notation},  we must have $G_1 \equiv 0 $. $G_2 \equiv 0$ implies $(g')^2 = K_2 t + N_2 $, which has a solution given by 
$$ g(t) = \frac{d - 2(K_2 t + N_2)^{3/2} }{3K_2 } . $$
Moreover, $G_1 \equiv 0$ implies $g^2(1 + (g')^2) = K_3 t + N_3 $. Thus, we have 
$$  (d - 2(K_2 t + N_2)^{3/2} )^2 - 9K_2^2K_3 t - 9K_2^2N_3 = 0, $$
which vanishes for every $t$ if $K_2 = 0 $.Therefore, $(g')^2 = N_2 > 0$. Unless $g$ is constant, which is not the case, then $g' = \pm \sqrt{N_2}$ and  $g^2(1 + N_2) = K_3 t + N_3 $. This implies $g^2 = K_4 t + N_4$, that is,
$$ 2gg' = \pm 2g \sqrt{N_2} =  K_4 , $$
which implies that $g$ is constant,  a contradiction. 

If $F_1 \equiv 0 $ we must have $G_2 \equiv 0$ or $F_2  \equiv 0$. Since $F_2 = (sf' - f)' \not\equiv 0$ by hypothesis, we must have  $G_2 \equiv 0 $. We also remember $g' \not\equiv 0$. If $((g')^2)'' = 0 $, then $(g')^2 = K t + L$. If $K = 0$, then $g(t) = \sqrt{L} t + b$. Hence, equation \eqref{eqsol1case} becomes
\begin{equation}\label{suave}
 -f''(\sqrt{L} t + b )^2 \sqrt{L}(1+L) =  2\sqrt{L}[L((f')^2 + 1 ) + 1] [(sf' - f) - t ] .\end{equation}
If $L = 0$, then $g$ is constant, a contradiction. Suppose  $L > 0$. Differentiating equation \eqref{suave} twice with respect to $t$ provides 
$$ -f''2\sqrt{L}L (1+L) = 0 , $$
which implies that $f'' =0 $, that is, $f(s) = c s + d $. 
Returning to equation \eqref{eqsol1case}, and differentiating with respect to $t$ gives
$$ 2\sqrt{L}[L(c^2 + 1 ) + 1] = 0 . $$
Since $L \ne 0 $ we have $ (c^2 + 1 ) = - 1/L $. A contradiction as $L > 0$.

If $K \ne 0$, we remember that as $F_1 \equiv 0$, then $ f''= M (1 + (f')^2)$. In addition, the solution of $(g')^2 = K t + L$, is given by
$$g(t) = \frac{b - 2(Kt + L )^{3/2}}{3K} . $$ 
If $M = 0 $ we have $f'' = 0$, a contradiction since $(sf' -f)' \not\equiv 0$. Thus $M \ne 0$. 
We return to equation \eqref{eqsol1case} with $(g')^2 = K t + L$ to obtain 
$$ -M((f')^2 + 1) g^2g'(1+(g')^2) + g^2 g'' (1+(f')^2)  =  2g'[(g')^2((f')^2 + 1 ) + 1] [(sf' - f) - t ]. $$
Since $g'\not \equiv 0 $, we divide both sides by $((f')^2 + 1 )$ and differentiate with respect to $s$ and then $t$ to obtain 
$$  2g'' (sf' - f)' = 2 \bigg[\frac{1}{1+(f')^2}\bigg]' $$
which implies that $g''$ is constant. Recalling that $g'(t) = \sqrt{Kt + L}$, we have  
$$g''(t) = \frac{1}{2} \frac{K}{\sqrt{Kt + L}}, $$
which implies that $K = 0 $, a contradiction. 
 
We now consider parameterizations of the second kind. Using equations \eqref{eqconditionsoliton}, \eqref{eq_normal_secondkind} and  \eqref{eq_meancurvature_secondcase_simplified},
\begin{equation}\label{eqsol2case}
-f''t^2 =  2[(f')^2  + 1] [s f' - f - b ] .
\end{equation}
Differentiating with respect to $t$ gives 
$$ -2f'' t= 0 . $$
Since $t > 0$, we have $f'' \equiv 0$ then $f(s) = cs + d$ and  \eqref{eqsol2case} becomes $d = - b$. In this case, the surface is contained in a totally geodesic hyperbolic plane. 
    
\end{proof}

\subsection{Proof of theorem \ref{theconfSolitonfirst}}
\begin{proof}
We begin by considering parameterizations of the first kind. Using equations \eqref{eqconditionconformalsoliton}, \eqref{eq_normal_firstkind} and \eqref{eq_meancurvature_firstcase_simplified}, we  have 
\begin{equation}\label{eqsolconf1case}
-f''g^2g'(1+(g')^2) + g^2g''(1+(f')^2)  =  -2(g+1)[(g')^2((f')^2 + 1 ) + 1] . 
\end{equation}
Suppose initially that $f''= 0 $. Then $f(s) = as + b $, and equation \eqref{eqsolconf1case} becomes 
$$ g^2g''(1+a^2)  + 2(g+1)[(g')^2(a^2 + 1 ) + 1] = 0 .$$
Since $g(t) > 0 $ for all $t$, by hypothesis, we have
$$     g''  + 2\frac{(g+1)}{g^2}(g')^2  + 2\frac{(g+1)}{g^2(1+a^2)}  = 0 .$$
The solution of such an ODE is described by Lemma \ref{theo_anal_theo_conf_sol_2}. Set $v(g) = (g')^2$, that is, $v'g' = 2g'g''$ and equation \eqref{eq_conf_teo_sol_1_anal_1} is given now as
$$ v' = - 4\frac{g+1}{g^2}v  - 4\frac{g+1}{g^2(1+a^2)}, $$
a first-order linear equation. 
Hence 
$$ (g')^2(t) = \frac{C e^{4/g(t)}}{g^4(t)}-  \frac{1}{1 + a^2}. $$
 
Suppose from now on that $f'f'' \not\equiv 0 $. We differentiate equation \eqref{eqsolconf1case} with respect to $s$ and divide both sides of the resulting equation by $f'f''$ to obtain 
\begin{equation}\label{eq_confsoliton_Mg}
    -\frac{f'''}{f'f''}g^2g'(1+(g')^2) + 2g^2g''  =  -4(g+1)(g')^2 . 
\end{equation}
Now, differentiating again with respect to $s$ gives 
$$\bigg[\frac{f'''}{f'f''}\bigg]'g^2g'(1+(g')^2) = 0 . $$
If $g^2g'(1+(g')^2) = 0 $, then $g' \equiv 0 $. Since $g (t) > 0 $ for all $t$, then equation \eqref{eqsolconf1case} becomes $ g(t) \equiv -1  $, a contradiction. Thus,  we must have 
\begin{equation}\label{eq_confsol_f'''}
     f''' = 2 M f'f'', \ \ \ M \in \mathbb{R} .
\end{equation}

If $M = 0 $, then $f(s) = as^2 + bs + c$ where $a,b,c \in \mathbb{R}$, with $a \ne 0$ . Thus, equation \eqref{eqsolconf1case} becomes   
$$ -2a g^2g'(1+(g')^2) + g^2g''(1+(2as + b )^2)  =  -2(g+1)[(g')^2((2as + b )^2 + 1 ) + 1]. $$
A differentiation with respect to $s$ and a simplification lead to $ g^2g''  =  -2(g+1)(g')^2 $. Set $v(g) = (g')^2 $, then  $v'g' = 2g'g''$ and as $g(t) > 0$ for all $t$. A first integration gives
$$ g'(t) = \pm \frac{C}{ e^{2g(t)}g^{2}(t)}, \ \ \ C > 0. $$
Returning to \eqref{eqsolconf1case} and since $g^2 g'' = -2 (g+1)(g')^2 $, we have
$$ \pm a C  (e^{4g}g^{4} + C)  = e^{6g}g^{4} (g + 1). $$
This implies that $g$ is constant, a contradiction.  
    
If $M \ne 0 $, then a first integration of \eqref{eq_confsol_f'''} gives $ f'' = M (f')^2 + B $. Thus, equation \eqref{eqsolconf1case} becomes
$$ (f')^2[ - M  g^2g'(1+(g')^2) + g^2g'' + 2(g+1)(g')^2 ] = [B g^2g' - 2(g+1)] (1 + (g')^2) -  g^2g'' . $$
By equation \eqref{eq_confsoliton_Mg} we have 
$$ [- B g^2g' + 2(g+1)](1 + (g')^2 ) = -  g^2g'' .  $$
The substitution in  equation \eqref{eq_confsoliton_Mg} gives 
\begin{equation}\label{eq_conf_sol_proof}
    (B - M) g^2g'(1+(g')^2) = 2(g+1) .
\end{equation}
If $B= M$ we have $g \equiv -1$, a contradiction. Since $g(t) > 0 $ for all $t$, we have  
$$ (B - M) g'(1+(g')^2) = 2(g+1)\frac{1}{g^2} . $$
The substitution in equation \eqref{eq_confsoliton_Mg} provides 
$$  g'' = - \frac{M}{M-B} \frac{g+1}{g^2} - 2\frac{g+1}{g^2} (g')^2 . $$
Set $v(g) = (g')^2$. Then $v'g' = 2g'g''$ and we have
$$  v' = - 4\frac{g+1}{g^2} \left(\frac{M}{2(M-B)}  + v \right). $$
The first integration of such an equation gives 
$$ (g')^2(t) = \frac{C e^{4/g(t)} - M }{2(M-B)g^4(t)} -  , \ \ \ C \ne 0 . $$
Returning to equation \eqref{eq_conf_sol_proof}, we obtain 
$$  (C e^{4/g} - M ) [(M-B)g^4 + C e^{4/g} - M]^2   = 4(g+1)^2(M-B)g^8, $$
which implies that $g $ is constant, a contradiction. 

We now consider parameterizations of the second kind. Using \eqref{eqconditionconformalsoliton} , \eqref{eq_normal_secondkind} and \eqref{eq_meancurvature_secondcase_simplified}, we have $ - t^2 f'' = 0 . $ Since $t > 0$, we have $f'' \equiv 0$, and this equation is contained in a totally geodesic hyperbolic plane.  

\end{proof}

\subsection{Proof of theorem \ref{Theo_general}}

Using Theorems \ref{theominfirst}, \ref{theoSolitonfirst}, and \ref{theconfSolitonfirst}, we can prove Theorem \ref{Theo_general} which summarizes all of them.

\begin{proof} 
Firstly, we recall what we observed at the end of Section 2. If $\alpha(s)=(a,s,1)$, then $\alpha(s) * \beta(t)$ is contained in a totally geodesic hyperbolic plane. Secondly, when $\beta(t)=(0,b,t)$, then we are in the second kind case and Theorems \ref{theominfirst}, \ref{theoSolitonfirst}, and \ref{theconfSolitonfirst}, state that the corresponding surface is contained, again, in a totally geodesic hyperbolic plane. Therefore, we can assume from now that $\alpha$ and $\beta$ are given as local graphs, i.e., our first kind case, and we investigate whether we can extend these graphs globally. It turns out that the only solution in all the three theorems is when $\alpha$ is a straight line segment and $\beta$ is given by the solution of equation \eqref{eq_sol_beta_minimal}, \eqref{eq_ODE_transsoliton_general theroem} or  \eqref{eq_soliton_JP_Main theorem}. Since $\alpha$ can be extended to a straight line defined in all $\mathbb{R}$, is enough to show that $\beta$ is complete in  $\mathbb{H}^3$ too. Since $\beta$ is defined in the context of Theorems \ref{theominfirst}, \ref{theoSolitonfirst} and \ref{theconfSolitonfirst}, we have the following cases
\begin{enumerate}
    \item  For the minimal case, consider $q$ a solution of equation \eqref{eq_sol_beta_minimal}  with initial conditions 
    $$ q(0)= y_0 > 0, \ \ \ m = \frac{y_0^4}{c^2 +1} , $$
    Then, by Lemma \ref{teoresultgfirstkind}, such a solution is defined in an interval $(-r,r)$, is concave and symmetric with respect to the $z$ axis, and attains a maximum at $t = 0 $, Additionally,  
    $$ \lim\limits_{t \to \pm r} q(t) = 0, \ \textnormal{ and } \ \lim\limits_{t \to \pm r} q'(t) = \pm \infty. $$  
    In this case we have the minimal translation cylinders (see Figure \ref{fig:Theo_min_1_func}). 
        
    \item For the hyperbolic translator, set $v = a + t$ and $(b^2 +1)^{-1} = k $ in equation \eqref{eq_soliton_JP_Main theorem}. Thus  
    $$  g''  =  - g'[(g')^2 + k] \frac{2v}{g^2}.$$
    The solution of such an ODE is described by Lemma \ref{lemma_soliton_JP}. In this case, we have hyperbolic grim reaper cylinders (see Figure~\ref{fig:Theo_sol_2}).

    \item For the conformal soliton case, suppose that $q$ is the solution of equation \eqref{eq_ODE_transsoliton_general theroem} with initial conditions
    $$ q_{\pm}(t_0) = y_0, \quad 
    m = \frac{y_0^4}{(1 + c^2)e^{4/y_0}}$$
    Hence, by Lemma \ref{theo_anal_theo_conf_sol_2} the solutions \( g_{\pm}(t) \) are bounded and defined for an interval $(0, t_0)$, \( g_+ \) is convex, \( g_- \) is concave, and for some constant $r < \infty $ the following limits hold:
    $$ \lim_{t \to t_0} g_{\pm}(t) = y_0, \quad \lim_{t \to 0} g_{\pm}(t) = y_0 \mp r, \quad \lim_{t \to 0} g_{\pm}'(t) = 0, \quad \lim_{t \to t_0} g_{\pm}'(t) = \pm \infty. $$
    In this case, we have the conformal grim reaper cylinders (see Figure \ref{fig:Theo_conf_sol_2}).
    \end{enumerate}
\end{proof}

\section{Funding}
T. A. Ferreira was supported by CNPq - Conselho Nacional de Desenvolvimento Cient\'ifico e Tecnol\'ogico grant number 162148/2021-6 and CAPES Print Process N° 88887.890396/2023-00. 

J. P. dos Santos was supported by FAPDF - Fundação de Apoio a Pesquisa do Distrito Federal, grant number 00193-00001678/2024-39, and CNPq, grant number 315614/2021-8.

\bibliographystyle{plain}
\bibliography{references}

\end{document}